\documentclass{cimart}

\usepackage{float}

\title{Some arithmetic properties of Weil polynomials of the form $t^{2g}+at^g+q^g$}

\author{Alejandro J. Giangreco Maidana}

\authorinfo{Facultad de Ingenier\'ia, Universidad Nacional de Asunci\'on, Paraguay and
Aix Marseille Univ, CNRS, I2M, Marseille, France}{agiangreco@ing.una.py}
 
\abstract{%
    An isogeny class $\mathcal{A}$ of abelian varieties defined over a finite field is said to be \emph{cyclic} if every variety in $\mathcal{A}$ has a cyclic group of rational points. In this paper we study the local cyclicity of Weil-central isogeny classes of abelian varieties, i.e.~those with Weil polynomials of the form $f_\mathcal{A}(t)=t^{2g}+at^g+q^g$, as well as the local growth of the groups of rational points of the varieties in $\mathcal{A}$ after finite field extensions.  We exploit the criterion: an isogeny class $\mathcal{A}$ with Weil polynomial $f$ is cyclic if and only if $f'(1)$ is prime with $f(1)$ divided by its radical.
    }

\keywords{
    abelian variety, Weil polynomial, group of rational points, cyclic, finite field
    }

\msc{
    11G10 (primary); 14G15, 14K15 (secondary)
    }

\VOLUME{34}
\ISSUE{2}
\YEAR{2026}
\NUMBER{2}
\DOI{https://doi.org/10.46298/cm.15946}

\begin{document}

\section{Introduction}
In this paper we study abelian varieties defined over finite fields with a cyclic group of rational points, and ``cyclic base field extensions''.
This subject is motivated by both applications and theory.
\begin{itemize}
\item 
Finite subgroups of abelian varieties over finite fields are suitable for multiple applications.
Cyclic subgroups of the group of rational points are used, for example, in cryptography, where the discrete logarithm problem is exploited. Abelian varieties can be very abstract objects. Jacobians of algebraic curves are abelian varieties and they are more tractable for application purposes.
\item
The statistics on cyclic varieties are related to Cohen-Lenstra heuristics (see \cite{CohenLenstra1984}), which, roughly speaking, states that random abelian groups tend to be cyclic. Historically, the question of cyclicity arose in the context of the conjectures of Lang and Trotter (see \cite{lang1977}):  given an elliptic curve defined over the rational numbers, we are interested in the set of primes such that the reduction is a cyclic elliptic curve. This question was studied also by Serre, Gupta, and Murty. Generalizations to higher dimensions were also done.
\end{itemize}

We restrict our study to abelian varieties with Weil polynomial that have the form $t^{2g}+at^g+q^g$. This includes elliptic curves, widely used in cryptography, such as isogeny-based cryptography. It also includes abelian surfaces with zero trace, ``almost all'' of them being isogenous to a principally polarizable abelian surface and to a Jacobian of a genus $2$ curve.

This leads to give the following:
\begin{definition}
Given an abelian variety $A$ defined over a finite field $k$, an isogeny class $\mathcal{A}$ of abelian varieties defined over $k$ and a rational prime $\ell$, we say that
\begin{enumerate}
    \item $A$ is \textbf{cyclic} if its group $A(k)$ of rational points is cyclic;
    \item $A$ is $\boldsymbol{\ell}$\textbf{-cyclic} if the  $\ell$-primary component $A(k)_\ell$ of its group $A(k)$ of rational points is cyclic;
    \item $\mathcal{A}$ is \textbf{cyclic} if the abelian variety $A$ is cyclic for all $A\in \mathcal{A}$;
    \item $\mathcal{A}$ is $\boldsymbol{\ell}$\textbf{-cyclic} if the abelian variety $A$ is $\ell$-cyclic for all $A\in \mathcal{A}$.
\end{enumerate}
\end{definition}

This paper concerns the cyclicity of isogeny classes. The Honda--Tate theory simplifies the study of isogeny classes by studying their Weil polynomials. Moreover, it is easy to verify the cyclicity of an isogeny class $\mathcal{A}$ given its Weil polynomial $f_\mathcal{A}$. Let $\widehat{n}$ denote the ratio of an integer $n$ to its radical, then, the cyclicity criterion is stated as follows.

\begin{theorem}[A. Giangreco, 2019, \cite{GIANGRECOMAIDANA2019139}]\label{th:weil_polynomial_criterion}
Let $\mathcal{A}$ be a $g$-dimensional $\mathbb{F}_q$-isogeny class of abelian varieties corresponding to the Weil polynomial $f_\mathcal{A}(t)$. Then $\mathcal{A}$ is cyclic if and only if  $f'_\mathcal{A}(1)$ is coprime with $\widehat{f_\mathcal{A}(1)}$.
\end{theorem}

This is in fact a local criterion (see Section \ref{subsec:local_cyc}) that can be easily deduced from the proof of Theorem \ref{th:weil_polynomial_criterion}. Asymptotic results about the cyclicity of abelian varieties were found in \cite{GIANGRECOMAIDANA2020101628}. Studying abelian varieties in all their generality is very complicated, so we will focus on the following family.

\begin{definition}
An isogeny class $\mathcal{A}$ of $g$-dimensional abelian varieties defined over the finite field $\mathbb{F}_q$ is said to be \textbf{Weil-central} if its Weil polynomial has the form
\[
f_\mathcal{A}(t)=t^{2g}+at^g+q^g.
\]
\end{definition}

In this paper we study the local cyclicity of Weil-central isogeny classes after base field extension as well as the local growth of their group of rational points. 
Given an abelian variety $A$ defined over the finite field $\mathbb{F}_q$ with $q$ elements, and belonging to an isogeny class $\mathcal{A}$, we denote by $\mathcal{A}_n$ the $\mathbb{F}_{q^n}$-isogeny class of $A$. 
For any prime number $\ell$, let $v_\ell(\cdot)$ be the usual $\ell$-adic valuation over the rational numbers.
Thus, for an $\ell$-cyclic isogeny class $\mathcal{A}$ we are interested in the following sets:
\begin{align*}
\mathfrak{g}_\ell (\mathcal{A}) &:= \{n \in \mathbb{N}:\;  v_{\ell}(f_{\mathcal{A}_n}(1)) > v_{\ell}(f_{\mathcal{A}}(1))\} \cup \{1\} \text{, and,}\\
\mathfrak{c}_\ell (\mathcal{A}) &:= \{n \in \mathbb{N}:\; \mathcal{A}_n \text{ is $\ell$-cyclic and } v_{\ell}(f_{\mathcal{A}_n}(1)) > v_{\ell}(f_{\mathcal{A}}(1))\} \cup \{1\}.
\end{align*}

The first set gives the extensions for which the $\ell$-primary component of the group of rational points is strictly bigger than the one over the base field. 
The second set gives the cyclic behavior of the $\ell$-component after such finite field extensions. 
Observe that we have 
\[
\mathfrak{c}_\ell (\mathcal{A})=  \mathfrak{g}_\ell (\mathcal{A}) \setminus \{n \in \mathbb{N}:\; \mathcal{A}_n \text{ is not $\ell$-cyclic} \}.
\]

For an integer $z$ we denote by $\omega_\ell(z)$ the order of $z$ in the multiplicative group $(\mathbb{Z}/\ell\mathbb{Z})^{\times}$, i.e. the smallest integer $m$ such that $z^m\equiv 1 \pmod{\ell}$. 
Then, our main result is stated as follows.
\begin{theorem}\label{thm:main}
Let $\ell$ be a prime and $\mathcal{A}$ be an $\ell$-cyclic Weil-central isogeny class of ordinary abelian varieties of dimension $g$ defined over $\mathbb{F}_q$.
Suppose that $\ell$ does not divide $g(q^g-1)$.
Let $S_{g,\ell}$ be the set of positive odd multiples of $\ell$ which are coprime with $g$.
Then, provided that $v_{\ell}(f_{\mathcal{A}}(1))\geq 1$, the set $\mathfrak{g}_\ell (\mathcal{A})$ contains $S_{g,\ell}$ and the set $\mathfrak{c}_\ell (\mathcal{A})$ contains the numbers in $S_{g,\ell}$ which are not divisible by $\omega_{\ell}(q^g)$.
\end{theorem}

\begin{remark}
From Lemma \ref{lemma:l_cyclicity} about the cyclicity, it follows that $v_{\ell}(f_{\mathcal{A}}(1))\geq 2$ implies that $\ell$ does not divide $q^g-1$, provided that the cyclicity hypothesis of Theorem \ref{thm:main} holds.
\end{remark}

\subsection*{Organization of the paper}
In Section \ref{sec:generalities} we recall some generalities about abelian varieties. Section \ref{sec:Weil-central} is devoted to the proof of Theorem \ref{thm:main}. It will be proved in different lemmas that can be useful by themselves.  We first study (Lemma \ref{lemma:a_n}) for which extensions a Weil-central isogeny class ``remains'' Weil-central. Then we study in Lemma \ref{lemma:growth} the growth behavior and we prove the assertion of Theorem \ref{thm:main} about the set $\mathfrak{g}_\ell (\mathcal{A})$.
Finally, we study the cyclicity and prove, as a consequence of Corollary \ref{cor:l-cyc}, the assertion of Theorem \ref{thm:main} about the set $\mathfrak{c}_\ell (\mathcal{A})$.
We discuss some examples in Section \ref{sec:examples}.

\section{Generalities on abelian varieties}\label{sec:generalities}
We refer the reader to \cite{mumford1970abelian} for the general theory of abelian varieties, and to \cite{Waterhouse1969} for abelian varieties over finite fields.

Let $q=p^r$ be a power of a prime, and let $k=\mathbb{F}_q$ be a finite field with $q$ elements. Let $A$ be an abelian variety of dimension $g$ over $k$. The set $A(k)$ of rational points of $A$ is a finite abelian group. It is the kernel of the endomorphism $1-F$, where $F$ is the well known Frobenius endomorphism of $A$. Multiplication by an integer $n$ is a group homomorphism whose kernel $A_n$ is a finite group scheme of rank $n^{2g}$. The group structure of the groups of points over $\overline{k}$ is known:
\begin{align}
\begin{split}\label{eqn:torsion_points}
A_n(\overline{k})&\cong (\mathbb{Z}/n\mathbb{Z})^{2g}, \qquad p\nmid n\\
A_p(\overline{k})&\cong (\mathbb{Z}/p\mathbb{Z})^i, \qquad 0\leq i \leq g.
\end{split}
\end{align}

For a fixed prime $\ell$ ($\neq p$), the $A_{\ell^n}$ form an inverse system under $A_{n+1} \overset{\ell}{\rightarrow} A_n$, and we can define \emph{the Tate module $T_\ell(A)$} by its inverse limit $\varprojlim A_{\ell^n}(\overline{k})$. This is a free $\mathbb{Z}_{\ell}$-module of rank $2g$ and the absolute Galois $\mathcal{G}$ group of $\overline{k}$ over $k$ operates on it by $\mathbb{Z}_{\ell}$-linear maps.

The Frobenius endomorphism $F$ of $A$ acts on $T_\ell(A)$ by a semisimple linear operator, and its characteristic polynomial $f_{A}(t)$ is called \emph{Weil polynomial of} $A$ (also called \emph{characteristic polynomial of} $A$). The Weil polynomial is independent of the choice of the prime $\ell$. Tate proved in \cite{tate1966} that a $k$-isogeny class $\mathcal{A}$ is determined by the Weil polynomial $f_A$ of any $A\in\mathcal{A}$, i.e. two abelian varieties $A$ and $B$ defined over $k$ are isogenous (over $k$) if and only if $f_A=f_B$. Thus the notation $f_{\mathcal{A}}$ is justified. If $\mathcal{A}$ is simple, $f_{\mathcal{A}}(t)=h_{\mathcal{A}}(t)^e$ for some irreducible polynomial $h_{\mathcal{A}}$. 

Weil proved that all of the roots of a Weil polynomial have absolute value $\sqrt{q}$ (they are called $q$\emph{-Weil numbers}). Thus, the Weil polynomial of an isogeny class $\mathcal{A}$ has the general form
\begin{align*}
    f_{\mathcal{A}}(t)=t^{2g}+a_1 t^{2g-1}+\dots + a_g t^g + a_{g-1} q t^{g-1}+\dots + a_1 q^{g-1} t + q^{g}.
\end{align*}

The cardinality of the group $A(k)$ of rational points of $A$ equals $f_{\mathcal{A}}(1)$, and thus it is an invariant of the isogeny class. An abelian variety $A$ is said to be \emph{ordinary} if it has a maximal $p$-rank. This is equivalent of having the central term $a_g$ of its Weil polynomial coprime with $p$. Thus, ordinariness it is also an invariant of the isogeny class.

Let $\{ \alpha_i \}_{i=1}^{2g}$ be the set of roots of the Weil polynomial of the abelian variety $A\in\mathcal{A}$ defined over $\mathbb{F}_q$. 
For a positive integer $n$, we denote by $\mathcal{A}_n$ the $\mathbb{F}_{q^n}$-isogeny class of the variety $A$ as defined over $\mathbb{F}_{q^n}$. If the polynomial $\prod (t-\alpha_i^n )$ has no repeated roots, then it corresponds to the Weil polynomial of $\mathcal{A}_n$.

\section{Weil-central isogeny classes}\label{sec:Weil-central}
As we stated, we are interested in Weil-central isogeny classes of abelian varieties. They have a Weil polynomial as follows
\begin{align}\label{WC_poly}
f_\mathcal{A}(t)=t^{2g}+at^g+q^g.    
\end{align}
Among these isogeny classes are those of elliptic curves and zero-trace abelian surfaces, with Weil polynomials $f_{\mathcal{E}} (t) =t^2 + at +q$ and $f_{\mathcal{S}} (t) =t^4 + at^2 +q^2$, respectively. Cyclicity of elliptic curves and their extensions were studied by Vl{\u{a}}du{\c{t}} in \cite{VLADUT199913} and \cite{VLADUT1999354}.

The following facts motivate the study of such isogeny classes. We know from \cite{Howe2008Principally} that among Weil-central isogeny classes of abelian surfaces, only such with Weil polynomial $t^4-qt^2+q^2$, and $p\equiv 1\pmod{3}$ do not contain a principally polarizable variety. Also, from \cite{Howe2009Jacobians}, very few do not contain the Jacobian of a $2$-genus curve.

Note that in order that a polynomial of the form (\ref{WC_poly}) to be a Weil polynomial, we need that $0\leq \mid a\mid \leq 2\sqrt{q^g}$. Indeed, if not, the polynomial $x^2+ax+q^g$ would have two different real roots, which would imply that (\ref{WC_poly}) has complex roots of absolute value different than $\sqrt{q}$. This also implies that the real part of the roots of $x^2+ax+q^g$ is exactly $-a/2$.

\textbf{Notations.} We denote simply by $(a,q)_g$ the Weil-central isogeny class $\mathcal{A}$ with Weil polynomial given by the expression (\ref{WC_poly}) above.
We denote by $N_{g,n}(a)$ the cardinalities of the groups of rational points of the varieties in $\mathcal{A}_n$, where $\mathcal{A}$ is defined by $(a,q)_g$. If $\mathcal{A}$ is clear from the context, we write $N_{g,n}$. We write $N$ instead of $N_{g,1}$ and $N_n$ instead of $N_{g,n}$ if the dimension $g$ is clear from the context. We recall that $N_{g,n}(\mathcal{A})=f_{\mathcal{A}_n}(1)$.

\subsection{Weil polynomial after field extension}
Our results on cyclicity and growth are for Weil-central isogeny classes.
Thus, in this section we study which extensions $\mathcal{A}_n$ of a  Weil-central isogeny class $\mathcal{A}$ are Weil-central as well.
Lemma \ref{lemma:a_n} below gives the answer and is from where the set $S_{g,\ell}$ of Theorem \ref{thm:main} contains only numbers which are coprime with $g$.

When dealing with cyclicity, it is natural to ask about the simplicity or not of the abelian variety.
For example, in the case of a Weil-central isogeny class $\mathcal{A}$ of abelian surfaces, we know that $\mathcal{A}_n$ splits for $n$ even (see \cite[Theorem 6]{HOWE2002139}).
However, our results about Weil-central isogeny classes (as well as our results on cyclicity) are independent on the simplicity or not of the considered abelian variety.

\begin{lemma}\label{lemma:a_n}
Suppose the isogeny class $\mathcal{A}$ has Weil polynomial $f_{\mathcal{A}}(t)=t^{2g} + a_1t^g+q^g$, with $\gcd(a_1,p)=1$, i.e. it is an ordinary isogeny class. Then, its extensions $\mathcal{A}_n$ have Weil polynomials $f_{\mathcal{A}_n}(t)=t^{2g} + a_nt^{g}+q^{ng}$ for every $n$ such that $\gcd(n,g)=1$, where $a_n$ is obtained recursively
\begin{align*}
        a_n&=(-1)^{n}a_1^n-\sum_{i=1}^{\lfloor n/2 \rfloor} \binom{n}{i}a_{n-2i}q^{gi}.
\end{align*}
\end{lemma}
\begin{proof}
We first prove that if $\mathcal{R}=\{\alpha_1,\dots,\alpha_g,q/\alpha_1,\dots,q/\alpha_g\}$ is the set of roots of $f_{\mathcal{A}}$, then
\[
\{\alpha_1^n,\dots,\alpha_g^n,(q/\alpha_1)^n,\dots,(q/\alpha_g)^n\}
\]
is the set of roots of $f_n(t) := t^{2g} + a_nt^{g}+q^{ng}$. Then we will prove that $f_n$ has no repeated roots if and only if $\gcd(n,g)=1$. Thus, for $\beta\in\mathcal{R}$, we will show that $\beta^n$ is a root of
\[
t^{2g} + a_nt^{g}+q^{ng}.
\]
It is clear that $\beta^n$ is a root of
\[
t^{2g} - (\beta^{ng}+(q/\beta)^{ng})t^{g}+q^{ng} \in \mathbb{C}[t].
\]
Thus we have to show that 
\[
a_n=- (\beta^{ng}+(q/\beta)^{ng}) \in \mathbb{Z}.
\]

In general, if we define $c_n=-x^n-(z/x)^n$ for $n>0$ and $c_0=-1$, we have that
\begin{align*}
    (-1)^n c_1^n &=(x + z/x)^n  \\
    &= x^n+(z/x)^n + \binom{n}{1} \left[x^{n-1}(z/x) +  x(z/x)^{n-1}\right] + \ldots\\
    &\quad \ldots +\binom{n}{i} \left[x^{n-i}(z/x)^i +  x^i(z/x)^{n-i}\right] + \ldots\\
    &\quad \ldots + \binom{n}{\lfloor n/2 \rfloor} A,
\end{align*}
where (observe that for $n$ odd we have that $\lfloor n/2 \rfloor = (n-1)/2$)
\begin{align*}
    A= x^{n/2}(z/x)^{n/2} \text{ or } A= x^{(n+1)/2}(z/x)^{(n-1)/2}+x^{(n-1)/2}(z/x)^{(n+1)/2}
\end{align*}
for $n$ even or odd, respectively. Equivalently, 
\[
A=z^{\lfloor n/2 \rfloor} (x+z/x)^{2(n/2- \lfloor n/2 \rfloor)}.
\]
Then
\begin{align*}
    (-1)^n c_1^n &=-c_n- \binom{n}{1}zc_{n-2} - \ldots - \binom{n}{i} z^i c_{n-2i} - \ldots - \binom{n}{\lfloor n/2 \rfloor} z^{\lfloor n/2 \rfloor} c_{\boldsymbol{\epsilon}},
\end{align*}
where $\boldsymbol{\epsilon}=0,1$ for $n$ even or odd, respectively. Finally
\begin{align*}
    c_n = (-1)^{n+1} c_1^n - \sum_{i=1}^{\lfloor n/2 \rfloor} \binom{n}{i}c_{n-2i}z^{i}.
\end{align*}

By taking $x=\alpha^g$ and $z=q^g$, we finish the first part of the proof.

Now we prove that $f_n(t)$ has no repeated roots if and only if $\gcd(n,g)=1$. 
For $k = 0,1,\dots,g-1$, let
\begin{equation*}
\theta/g+k2\pi/g, \quad\text{and}\quad -\theta/g-k2\pi/g,
\end{equation*}
be the arguments of the complex roots in $\mathcal{R}$ of $f_{\mathcal{A}}$, respectively, where $\theta$ and $-\theta$ are the arguments of the roots of $t^2+a_1 t+q^g$, and $0 <\theta < \pi$.
The polynomial $f_n$ has repeated roots if and only if two of the roots in $\mathcal{R}$ to the $n$-th power are equal. We have
\begin{enumerate}
\item 
either $\alpha_i^n=\alpha_j^n$ for some $i\neq j\in \{0,\dots,g-1\}$. Then 
\[
\frac{2\pi n}{g}(i-j)\equiv 0 \pmod{2\pi},
\]
which holds if and only if $n(i-j)/g\in \mathbb{Z}$ if and only if $\gcd(n,g)>1$;
\item
either $\alpha_i^n=(q^g/\alpha_j)^n$ for some $i,j\in \{0,\dots,g-1\}$. Then
\[
\frac{n}{g}[2\theta + 2\pi(i+j)]\equiv 0 \pmod{2\pi}.
\]
A necessary condition is that $\theta\in\pi \mathbb{Q}$. Note that $\cos(\theta)=-a_1/(2\sqrt{q^g})$. On one hand, it can be shown (from \cite[Thm. 1]{Lehmer1933}, for example) that $2\cos(\theta)$ is an algebraic integer if $\theta\in\pi \mathbb{Q}$. On the other hand, $a_1/\sqrt{q^g}$ is an algebraic number of degree $\leq 2$ which is not an algebraic integer since $\gcd(a_1,p)=1$. Thus, this would be a contradiction.
\end{enumerate}

The Honda-Tate theory for ordinary abelian varieties ensures that the polynomial $f_n(t)=t^{2g} + a_nt^g+q^g$ obtained is indeed the Weil polynomial $f_{\mathcal{A}_n}$ of the extension $\mathcal{A}_n$.
\end{proof}

\subsection{Local growth}
Given an $\ell$-cyclic isogeny class $\mathcal{A}$ (with $\ell \mid f_{\mathcal{A}}(1)$, i.e. with non trivial $\ell$-primary component) it is clear that for all $n$, $\ell \mid f_{\mathcal{A}_n}(1)$ since $A(\mathbb{F}_q)\subset A(\mathbb{F}_{q^n})$.
Thus, it is more interesting to know for which of these values of $n$ the $\ell$-part increases (relatively to the base field). Lemma \ref{lemma:growth} below gives an answer. 

We first fix a polynomial that will be useful. For every positive integer $n$, we set
\begin{align*}
    P_n (x) := \sum_{i=0}^{n-1}x^i. 
\end{align*}
Note that $(x-1)P_n(x)=x^n -1$. Notice that Lemma \ref{lemma:growth} below is only valid for $n$ odd (so only odd integers are considered in Theorem \ref{thm:main}). We write first the polynomial $P_n$  in a convenient way:
\begin{lemma}
For $n$ odd, the polynomial $P_n(x)$ can be obtained recursively:
\[
P_n(x)=(x +1)^{n-1} - \sum_{i=1}^{(n-1)/2} \left[\binom{n}{i} -2 \binom{n-1}{i-1}\right]x^i P_{n-2i}(x),
\]
with $P_1(x)=1$.
\end{lemma}
\begin{proof}
The proof is straightforward by using induction on $n$ and showing directly that the equality $(x-1)P_n(x)=(x-1)$``right-hand-side'' holds.
\end{proof}

\begin{lemma}\label{lemma:growth}
For every positive odd integer $n$ and any prime integer $\ell$, we have 
\[
v_\ell(N_{n})\geq \min\{2v_\ell(N_{1}), v_\ell(N_{1}) + v_\ell(nP_n(q^g))\},
\]
provided that $\ell \mid N_{1}$.
\end{lemma}
\begin{proof}
We suppose $n$ odd. Recall that for Weil-central isogeny classes 
\begin{align*}
    N_{g,1}&=q^g+a_1+1, \text{ and},\\
    N_{g,n}&=q^{gn}+a_n+1,
\end{align*}
where $a_n$ can be computed by using Lemma \ref{lemma:a_n}.
From the hypothesis $v_{\ell} (N_{1}):= m>0$ we have
\begin{align*}
    N_{g,1} \equiv q^{g}+a_1+1 \equiv z\ell^{m} \pmod{\ell^{2m}}, \quad 0<z<\ell^m, \ell\nmid z.
\end{align*}
From now, all congruences are modulo $\ell^{2m}$.
First, we show by induction on $n$ that: 
\begin{align*}
    a_n &\equiv -q^{gn}-1+z\ell^{m} n P_n(q^g).
\end{align*}
For $n=1$,
\begin{align*}
    a_1 &\equiv -q^{g}-1+z\ell^{m} P_1(q^g),
\end{align*}
with $P_1=1$.

Using the induction hypothesis for $i=1,\dots, (n-1)/2$ (so that $n-2i<n$), we have that
\begin{align*}
a_{n-2i} q^{gi} 
     &\equiv \left[-q^{g(n-2i)}-1+z\ell^{m} (n-2i) P_{n-2i}(q^g)\right] q^{gi} \\
    &\equiv -q^{gn-gi}-q^{gi}+z\ell^{m} q^{gi} (n-2i) P_{n-2i}(q^g),
\end{align*}
then taking the sum over $i=1,\dots, (n-1)/2$
\begin{align*}
\sum \binom{n}{i} a_{n-2i} q^{gi} 
     &\equiv \sum_{i=1}^{(n-1)/2} \binom{n}{i} \left[ -q^{gn-gi}-q^{gi}+z\ell^{m} q^{gi} (n-2i) P_{n-2i}(q^g)\right] \\
     &\equiv -(q^g +1)^n +q^{gn} +1 +z\ell^{m}\sum_{i=1}^{(n-1)/2} \binom{n}{i} (n-2i) \left[q^{gi} P_{n-2i}(q^g)\right]. 
\end{align*}
\enlargethispage{\baselineskip}

From Lemma \ref{lemma:a_n}, and using the fact that for $m>0$, \[(x+y\ell^m)^n \equiv x^n+nx^{n-1} y\ell^m \pmod{\ell^{2m}}\]
at the step marked with $(*)$, we have
{\small
\begin{align*}
a_n &\equiv a_1^n - \sum \binom{n}{i} a_{n-2i} q^{gi} \\
    &\equiv [-(q^g +1)+z\ell^m]^n -  \Big[ -(q^g +1)^n +q^{gn} +1 +z\ell^{m}\sum_{i=1}^{(n-1)/2} \binom{n}{i} (n-2i) \left[q^{gi} P_{n-2i}(q^g)\Big] \right] \\
    &\stackrel{(*)}{\equiv}  -(q^g +1)^n + n(q^g +1)^{n-1} z\ell^m \\ &\qquad-  \Big[ -(q^g +1)^n +q^{gn} +1 +z\ell^{m}\sum_{i=1}^{(n-1)/2} \binom{n}{i} (n-2i) \left[q^{gi} P_{n-2i}(q^g)\Big] \right] \\
    &\equiv - q^{gn} - 1 + z\ell^{m} \Big[ n(q^g +1)^{n-1} - \sum_{i=1}^{(n-1)/2} \binom{n}{i} (n-2i) \left[q^{gi} P_{n-2i}(q^g)\right] \Big]  \\
    &\equiv - q^{gn} - 1 + z\ell^{m} n\underbrace{\Bigg[ (q^g +1)^{n-1} - \sum_{i=1}^{(n-1)/2} \bigg[\binom{n}{i} -2 \binom{n-1}{i-1}\bigg] \bigg[q^{gi} P_{n-2i}(q^g)\bigg] \Bigg] }_{P_n(q^g)}, 
\end{align*}
}
which completes the induction part. Then we compute $N_n$:
\begin{align*}
    N_n &\equiv q^{gn}+a_n+1 \\
    &\equiv  q^{gn} +\left[ - q^{gn} - 1 + z\ell^{m}nP_n(q^g) \right] + 1\\
    &\equiv z\ell^{m}nP_n(q^g).
\end{align*}
Finally, we have $N_n=z\ell^{m}nP_n(q^g) + s\ell^{2m}$, for some integer $s$. The result follows.
\end{proof}

Since $v_\ell(nP_n(q^g))\geq 1$ if $\ell \mid n$,  we have then proved the assertion about $\mathfrak{g}_\ell (\mathcal{A})$ of Theorem \ref{thm:main}. 

\begin{remark}
From Lemma \ref{lemma:growth} above, we also have that $\mathfrak{g}_\ell (\mathcal{A}) \supset \{n\in\mathbb{N} : \ell \mid P_n(q^g)\}$. However, from Lemma \ref{lemma:l_cyclicity} (see Section \ref{subsec:local_cyc} below), $\ell \mid P_n(q^g)$ implies $\ell \mid (q^g)^{n}-1$, then $\mathcal{A}_n$ is not necessarily $\ell$-cyclic.
\end{remark}

\subsection{Local cyclicity}\label{subsec:local_cyc}
In this section we will give the characterization of the local cyclicity of Weil-central isogeny classes and their extensions. 
Following the definition of cyclicity for isogeny classes, a local version of \cite[Theorem 2.2]{GIANGRECOMAIDANA2019139} gives the \emph{cyclicity criterion}: for any prime $\ell$ and for any isogeny class $\mathcal{A}$ we have that
\begin{align}\label{thm:l-cyclic}
    \mathcal{A} \text{ is } \ell\text{-cyclic if and only if } \ell\text{ does not divide } \gcd(\widehat{f_\mathcal{A}(1)},f_\mathcal{A}'(1)).
\end{align}
This has a meaning only if $A(k)_\ell$ is not trivial for some $A\in\mathcal{A}$ (and thus for all $A\in\mathcal{A}$), equivalently if $\ell$ divides $f_\mathcal{A} (1)$. 
Sometimes we use a weaker version of the cyclicity criterion, namely $\ell\nmid \gcd(f_\mathcal{A}(1),f_\mathcal{A}'(1))$ implies that the isogeny class is $\ell$-cyclic.
We give a complete description of the local cyclicity:
\begin{lemma}\label{lemma:l_cyclicity}
Given a Weil-central isogeny class $(a,q)_g$ and a rational prime $\ell$:
\begin{enumerate}
    \item if $\ell \nmid g$ and $\ell \nmid q^g-1$, then $(a,q)_g$ is $\ell$-cyclic;
    \item if $\ell \nmid g, \ell \mid q^g-1$ and $\ell \mid f(1)$, then $(a,q)_g$ is $\ell$-cyclic if and only if $\ell^2 \nmid f(1)$;
    \item if $\ell \mid g$, then $(a,q)_g$ is $\ell$-cyclic if and only if $\ell^2\nmid f(1)$.
\end{enumerate}
\end{lemma}
\begin{proof}
Recall that $(a,q)_g$ corresponds to the isogeny class with Weil polynomial of the form $f(t)=t^{2g}+at^g+q^g$. Then we have
\begin{align*}
    f(1)&=1+a+q^g=(q^g-1)+(a+2),\\
    f'(1)&=g(2+a).
\end{align*}
Then we prove the lemma point by point using the \emph{cyclicity criterion} (\ref{thm:l-cyclic}), which is equivalent to
\begin{align}\label{thm:not-l-cyclic}
    \mathcal{A} \text{ is not } \ell\text{-cyclic if and only if } \ell^2 \mid f(1) \text{ and } \ell \mid f'(1).
\end{align}
\begin{enumerate}
\item If we suppose that $(a,q)_g$ is not $\ell$-cyclic, then $\ell^2 \mid (q^g-1)+(a+2)$ and $\ell \mid a+2$.
This implies $\ell \mid q^g-1$, a contradiction.
\item We prove the assertion by contraposition:
\begin{align*}
(a,q)_g  \text{ is not } \ell\text{-cyclic} &\Leftrightarrow \ell^2 \mid f(1), \: \ell \mid f'(1) &&\text{from criterion (\ref{thm:not-l-cyclic})}\\
&\Leftrightarrow \ell^2 \mid f(1), \: \ell \mid a+2  &&\text{since}\; \ell \nmid g\\
&\Leftrightarrow \ell^2 \mid f(1) &&\text{since}\; \ell \mid q^g-1.
\end{align*}
\item 
In the case $\ell \mid g$, we have that $\ell \mid f'(1)$. The result can be deduced from the contrapositive version (\ref{thm:not-l-cyclic}) of the cyclicity criterion.
\end{enumerate}
\end{proof}

As we consider here Weil-central isogeny classes such that the $\ell$-part grows, this implies in particular that $\ell^2 \mid f_n(1)$. Thus, Lemma \ref{lemma:l_cyclicity} says that the local cyclicity at a prime $\ell$ is possible only if $\ell\nmid g$ and $\ell\nmid (q^n)^g-1$.

\begin{corollary}\label{cor:l-cyc}
Suppose $\ell\nmid g$ and $\ell\nmid q^g-1$ (in particular $\mathcal{A}_1$ is $\ell$-cyclic), then
\begin{align*}
    \{n \in \mathbb{N}: \mathcal{A}_n \text{ is $\ell$-cyclic}\}
\end{align*}
contains the set of positive integers with no common factor with $g$ and which are not multiple of $\omega_\ell (q^g)$.
\end{corollary}
\begin{proof}
In order to have that $\mathcal{A}_n$ is Weil-central, we do not consider $n$ with a common factor with $g$ (see Lemma \ref{lemma:a_n}).
For the $\ell$-cyclicity of $(a_n,q^n)_g$ we use item (1) of Lemma \ref{lemma:l_cyclicity}. Observe that in this case, the cyclicity is independent of the value $a_n$. We set $\delta:=\omega_\ell(q^g)$ and we consider the Euclidean division
\[
n=c\delta +r , \; 0\leq r < \delta.
\]
Then we have
\begin{align*}
q^{gn}\equiv q^{g(c\delta +r)} \equiv q^{gc\delta} q^{gr} \equiv q^{gr} \pmod{\ell},
\end{align*}
which is congruent to $1$ if and only if $r$ is zero, if and only if $n$ is a multiple of $\delta$.
\end{proof}
Observe that here we do not consider the growth of the $\ell$-part, only the cyclicity. Moreover, the isogeny classes $\mathcal{A}_n$ can be ``$\ell$-trivial''.
We have then proved the assertion about $\mathfrak{c}_\ell (\mathcal{A})$.
Finally, this completes the proof of Theorem \ref{thm:main}.

\section{Examples}\label{sec:examples}
Consider the ordinary elliptic curve ($g=1$) defined by the Weil polynomial
\[
f_{\mathcal{E}} (t) =t^2 + t + 73.
\]
We have $f_{\mathcal{E}} (1) =75=3\times 5^2$ and $q^g-1=72$ is not divisible by $5$. Then the isogeny class $\mathcal{E}=(1,73)_1$ is $5$-cyclic and so the conditions of Theorem \ref{thm:main} are verified. We also have $\omega_{5}(73)=4$. Consequently 
\begin{gather*}
    \mathfrak{g}_5 (\mathcal{E}) \supset \{n\in \mathbb{N}: 5\mid n, 2\nmid n\}, \\ 
    \mathfrak{c}_5 (\mathcal{E}) \supset \{n\in \mathbb{N}: 5\mid n, 2\nmid n, 4\nmid n\} = \{n\in \mathbb{N}: 5\mid n, 2\nmid n\}.
\end{gather*}
Note that the elliptic curve $\mathcal{E}$ is $3$-cyclic since $v_3(f_{\mathcal{E}} (1))=1$. However, we cannot apply Theorem \ref{thm:main} for $\ell=3$ since $q^g-1=72$ is divisible by $3$. 
This and another few examples are listed in Table \ref{table_exple}.

\begin{table}[h]
\footnotesize
\begin{center}
\begin{tabular}{@{}ccccll@{}}
\hline
$(a,q)_g$ & $f_{\mathcal{A}}(1)$ & $\ell$ & $\omega_{\ell}(q^g)$ & \qquad $\subset \mathfrak{g}_\ell (\mathcal{A})$ & \qquad $\subset \mathfrak{c}_\ell (\mathcal{A})$\\
\hline
$(1,73)_1$  & $3\times 5^2$ & $5$ & $4$ & $ \{n\in \mathbb{N}: 5\mid n, 2\nmid n\}$ & $ \{n\in \mathbb{N}: 5\mid n, 2\nmid n\}$\\
$(11,17)_3$  & $5^2 \times 197$ & $5$ & $4$ & $ \{n\in \mathbb{N}: 5\mid n, 2\nmid n, 3\nmid n\}$ & $ \{n\in \mathbb{N}: 5\mid n, 2\nmid n, 3\nmid n\}$\\ 
$(17,19)_3$ & $13\times 23^2$ & $23$ & $22$ & $ \{n\in \mathbb{N}: 23\mid n, 2\nmid n, 3\nmid n\}$ & $ \{n\in \mathbb{N}: 23\mid n, 2\nmid n, 3\nmid n, 22\nmid n\}$\\
$(20,7)_6$ & $70\times 41^2$ & $41$ & $20$ & $\{n\in \mathbb{N}: 41\mid n, 2\nmid n, 3\nmid n\}$ & $\{n\in \mathbb{N}: 41\mid n, 2\nmid n, 3\nmid n, 20\nmid n\}$\\
\hline
\end{tabular}
\end{center}
\caption{Examples related to Theorem \ref{thm:main}}\label{table_exple}%
\end{table}
Note that in the last two rows of Table \ref{table_exple}, the parts $22\nmid n$ and $20\nmid n$ are unnecessary. For all cases, the sets obtained and included in $\mathfrak{g}_\ell (\mathcal{A})$ and $\mathfrak{c}_\ell (\mathcal{A})$, respectively, are the same. Table \ref{table_exple_val} shows the valuations $v_\ell(\mathcal{A}_5), v_\ell(\mathcal{A}_{15}), v_\ell(\mathcal{A}_{25}),\dots$ for the elliptic curve case, as well as the valuations for the other examples of Table \ref{table_exple}.

\begin{table}[H]
\footnotesize
\begin{center}
\begin{tabular}{cccccccccccccccccccccc}
\hline
$(a,q)_g$ & $\ell$ & $v_\ell(\mathcal{A})$ & \multicolumn{18}{c}{$v_\ell(\mathcal{A}_n)$} & Last $n$ \\
\hline
$(1,73)_1$  & 5 & $2$ & $3$ & $3$ & $4$ & $3$ & $3$ & $3$ & $3$ & $4$ & $3$ & $3$ & $3$ & $3$ & $5$ & $3$ & $3$ & $3$ & $3$ & $4$ & $175$\\ \hline
$(11,17)_3$  & 5 & $2$ & $3$ & $4$ & $3$ & $3$ & $3$ & $3$ & $3$ & $3$ & $5$ & $3$ & $3$ & $4$ & $3$ & $3$ & $3$ & $3$ & $3$ & $3$ & $265$\\ \hline
$(17,19)_3$ & $23$ & $2$ & $4$ & $4$ & $4$ & $6$ & $4$ & $4$ & $4$ & $6$ & $4$ & $4$ & $4$ & $4$ & $4$ & $4$ & $4$ & $4$ & $4$ & $4$ & $1219$\\ \hline
$(20,7)_6$ & $41$ & $2$ & $3$ & $3$ & $3$ & $3$ & $3$ & $3$ & $3$ & $3$ & $3$ & $3$ & $3$ & $3$ & $3$ & $4$ & $3$ & $3$ & $3$ & $3$ & $2173$\\ 
\hline
\end{tabular}
\end{center}
\caption{Valuation $v_\ell(\mathcal{A}_n)$ for $n$ in the sets of Table \ref{table_exple} }\label{table_exple_val}%
\end{table}

Consider the example given previously $\mathcal{A}_1=(11,17)_3$, that is, defined by the Weil polynomial
\[
f_{\mathcal{A}} (t) =t^{6} +11t^3 + 17^3.
\]
Table \ref{table_case_complet} shows information about extensions $\mathcal{A}_n$ for $n$ up to $10$.
As expected from (the proof of) Lemma \ref{lemma:a_n}, when $n$ and $g$ have a common factor, the isogeny class $\mathcal{A}_n$ is not Weil-central. 
The Weil polynomial of $\mathcal{A}_n$ is of the form $h^e$, with $h$ irreducible and ordinary. So $\mathcal{A}_n$ is the power of a lower dimensional class of abelian varieties, which are Weil central. We represent them as $(a,q)_{g'}^e$ in Table \ref{table_case_complet}.
For example, $\mathcal{A}_3$ is the product of three copies of the isogeny class of elliptic curves with Weil polynomial $t^{2} +11t + 17^3$.
Theorem \ref{thm:main} says that the $5$-part grows with respect to the base field for $n=5$. Besides that, we see that for $n \in \{4,8,10\}$ 
the $5$-part grows as well. Among these extensions, $\mathcal{A}_n$ is $5$-cyclic for $n=5$ (Theorem \ref{thm:main}). It is not $5$-cyclic for $n=4$, and thus for $n=8$ (since $\mathcal{A}_8$ is an extension of $\mathcal{A}_4$), but it is $5$-cyclic for $n=10$.

\begin{table}[h]
\footnotesize
\begin{center}
\begin{tabular}{@{}cccccc@{}}
$n$ & $\mathcal{A}_n$ & Simple? & Weil-central? & $v_5(f_n(1))$ & $5$-cyclic?\\ \hline \hline
$1$ & $(11,17^1)_3$ &    {\color{ForestGreen} Yes} & {\color{ForestGreen} Yes} & $2$ & {\color{ForestGreen} Yes}\\ \hline
$2$ & $(9705,17^2)_3$ & {\color{ForestGreen} Yes} & {\color{ForestGreen} Yes} & $2$ & {\color{ForestGreen} Yes}\\ \hline
$3$ & $(11,17^3)_1^3$ & {\color{BrickRed} No} & {\color{BrickRed} No} & $2$ & {\color{BrickRed} No}\\ \hline
$4$ & $(-45911887,17^4)_3$ & {\color{ForestGreen} Yes} & {\color{ForestGreen} Yes} & $3$ & {\color{BrickRed} No}\\ \hline
$5$ & $(1295031331,17^5)_3$ & {\color{ForestGreen} Yes} & {\color{ForestGreen} Yes} & $3$ & {\color{ForestGreen} Yes}\\ \hline
$6$ & $(9705,17^6)_1^3$ & {\color{BrickRed} No} & {\color{BrickRed} No} & $2$ & {\color{BrickRed} No}\\ \hline
$7$ & $(-8687006247293,17^7)_3$ & {\color{ForestGreen} Yes} & {\color{ForestGreen} Yes} & $2$ & {\color{ForestGreen} Yes}\\ \hline
$8$ & $(-942656893441247,17^8)_3$ & {\color{ForestGreen} Yes} & {\color{ForestGreen} Yes} & $3$ & {\color{BrickRed} No}\\ \hline
$9$ & $(-160798,17^9)_1^3$ & {\color{BrickRed} No} & {\color{BrickRed} No} & $2$ & {\color{BrickRed} No}\\ \hline
$10$ & $(4047739954748000025,17^{10})_3$ & {\color{ForestGreen} Yes} & {\color{ForestGreen} Yes} & $3$ & {\color{ForestGreen} Yes}\\ \hline
\hline
\end{tabular}
\end{center}\caption{$\mathcal{A}_n$ for $n$ up to $10$, where $\mathcal{A}_1=(11,17)_3$}\label{table_case_complet}%
\end{table}

\subsubsection*{Acknowledgments} 
I would like to express my gratitude to the anonymous referee for the relevant comments that led to this improved version of the article.


\EditInfo{June 27, 2025}{September 19, 2025}{Tiago Macedo, Jaqueline Mesquita, Mariel Saez and Rafael Potrie}


{\small
\begin{thebibliography}{10}

    \bibitem{CohenLenstra1984}
    H.~Cohen and H.~W. Lenstra, Jr.
    \newblock Heuristics on class groups of number fields.
    \newblock In {\em Number theory, {N}oordwijkerhout 1983 ({N}oordwijkerhout,
      1983)}, volume 1068 of {\em Lecture Notes in Math.}, pages 33--62. Springer,
      Berlin, 1984.
    
    \bibitem{GIANGRECOMAIDANA2019139}
    A.~J. Giangreco-Maidana.
    \newblock On the cyclicity of the rational points group of abelian varieties
      over finite fields.
    \newblock {\em Finite Fields Appl.}, 57:139--155, 2019.
    
    \bibitem{GIANGRECOMAIDANA2020101628}
    A.~J. Giangreco-Maidana.
    \newblock Local cyclicity of isogeny classes of abelian varieties defined over
      finite fields.
    \newblock {\em Finite Fields Appl.}, 62:101628, 9, 2020.
    
    \bibitem{Howe2008Principally}
    E.~W. Howe, D.~Maisner, E.~Nart, and C.~Ritzenthaler.
    \newblock Principally polarizable isogeny classes of abelian surfaces over
      finite fields.
    \newblock {\em Math. Res. Lett.}, 15(1):121--127, 2008.
    
    \bibitem{Howe2009Jacobians}
    E.~W. Howe, E.~Nart, and C.~Ritzenthaler.
    \newblock Jacobians in isogeny classes of abelian surfaces over finite fields.
    \newblock {\em Ann. Inst. Fourier (Grenoble)}, 59(1):239--289, 2009.
    
    \bibitem{HOWE2002139}
    E.~W. Howe and H.~J. Zhu.
    \newblock On the existence of absolutely simple abelian varieties of a given
      dimension over an arbitrary field.
    \newblock {\em J. Number Theory}, 92(1):139--163, 2002.
    
    \bibitem{lang1977}
    S.~Lang and H.~Trotter.
    \newblock Primitive points on elliptic curves.
    \newblock {\em Bull. Amer. Math. Soc.}, 83(2):289--292, 1977.
    
    \bibitem{Lehmer1933}
    D.~H. Lehmer.
    \newblock Questions, {D}iscussions, and {N}otes: {A} {N}ote on {T}rigonometric
      {A}lgebraic {N}umbers.
    \newblock {\em Amer. Math. Monthly}, 40(3):165--166, 1933.
    
    \bibitem{mumford1970abelian}
    D.~Mumford.
    \newblock {\em Abelian varieties}, volume~5 of {\em Tata Institute of
      Fundamental Research Studies in Mathematics}.
    \newblock Tata Institute of Fundamental Research, Bombay; by Oxford University
      Press, London, 1970.
    
    \bibitem{tate1966}
    J.~Tate.
    \newblock Endomorphisms of abelian varieties over finite fields.
    \newblock {\em Invent. Math.}, 2:134--144, 1966.
    
    \bibitem{VLADUT199913}
    S.~G. Vl\u{a}du\c{t}.
    \newblock Cyclicity statistics for elliptic curves over finite fields.
    \newblock {\em Finite Fields Appl.}, 5(1):13--25, 1999.
    
    \bibitem{VLADUT1999354}
    S.~G. Vl\u{a}du\c{t}.
    \newblock On the cyclicity of elliptic curves over finite field extensions.
    \newblock {\em Finite Fields Appl.}, 5(4):354--363, 1999.
    
    \bibitem{Waterhouse1969}
    W.~C. Waterhouse.
    \newblock Abelian varieties over finite fields.
    \newblock {\em Ann. Sci. \'Ecole Norm. Sup. (4)}, 2:521--560, 1969.

\end{thebibliography}
}
\end{document}